\documentclass[11pt,a4paper]{amsart}

\usepackage[latin1]{inputenc}
\usepackage[english]{babel}
\usepackage{amsmath}

\usepackage{amssymb, mathabx}
\usepackage{graphicx}
\usepackage{braket}
\usepackage[pdftex,plainpages=false,colorlinks,hyperindex,bookmarksopen,linkcolor=red,citecolor=blue,urlcolor=blue]{hyperref}

\usepackage{cite}

\usepackage{mathrsfs}

\usepackage{epstopdf}

\usepackage{braket}

\usepackage[hmargin=3cm,vmargin={3.5cm,4cm}]{geometry}

\theoremstyle{theorem}
\newtheorem{thm}{Theorem}

\theoremstyle{definition}                                 

\theoremstyle{definition}                           

\theoremstyle{remark}                             
\newtheorem*{rmk}{Remark}              

\usepackage{color}

\usepackage{mathtools,slashed}

\newcommand{\be}{\begin{eqnarray}}
\newcommand{\ee}{\end{eqnarray}}

\newcommand{\R}{\mathbb{R}}  
\newcommand{\C}{\mathbb{C}} 
\newcommand{\N}{\mathbb{N}} 
\def\eg{{\it e.g. }} 
\def\ie{{\it i.e. }}

\newcommand{\wt}[1]{\widetilde{#1}}

\def\eg{{\it e.g.}\ }
\def\ie{{\it i.e.}\ }

\numberwithin{equation}{section}

\allowdisplaybreaks

\begin{document}
\title{A comment on some new definitions of fractional derivative}
	
	    \author{Andrea Giusti$^\dagger$}
		\address{${}^\dagger$ Department of Physics $\&$ Astronomy, University of 	
    	    Bologna and INFN. Via Irnerio 46, Bologna, ITALY and 
	    	 Arnold Sommerfeld Center, Ludwig-Maximilians-Universit\"at, 
	    	 Theresienstra{\ss}e~37, 80333 M\"unchen, GERMANY.}	
 		\email{agiusti@bo.infn.it}
 

	\thanks{ }	
	
    \date  {\today}

\begin{abstract}
After reviewing the definition of two differential operators which have been recently introduced by Caputo and Fabrizio and, separately, by Atangana and Baleanu, we present an argument for which these two integro-differential operators can be understood as simple realizations of a much broader class of fractional operators, \ie the theory of Prabhakar fractional integrals. Furthermore, we also provide a series expansion of the Prabhakar integral in terms of Riemann-Liouville integrals of variable order. Then, by using this last result we finally argue that the operator introduced by Caputo and Fabrizio cannot be regarded as fractional. Besides, we also observe that the one suggested by Atangana and Baleanu is indeed fractional, but it is ultimately related to the ordinary Riemann-Liouville and Caputo fractional operators. All these statements are then further supported by a precise analysis of differential equations involving the aforementioned operators. To further strengthen our narrative, we also show that these new operators do not add any new insight to the linear theory of viscoelasticity when employed in the constitutive equation of the Scott-Blair model.
\end{abstract}

    \maketitle


\section{Introduction}
	The mathematical literature, mostly in the last three years, has seen an intensive inflation of papers proposing some ``new'' or ``alternative'' notions of fractional derivatives.
	
	In this letter we aim to discuss the connection between two of the most known new operators, namely the Caputo-Fabrizio (CF) operator and the regularized Atangana-Baleanu operator (ABC), with the theory of Prabhakar fractional integrals.
	
	First, let us start by recalling the definition of the CF operator \cite{Caputo-Fabrizio},
	\be \label{CF-derivative} 
	^{CF} \textbf{D} _{a+} ^\alpha f (t) = 
	\frac{M(\alpha)}{1 - \alpha} \int _a ^t  \exp \left[ - \frac{\alpha}{1 - \alpha} \, (t - \tau) \right] \, f' (\tau) \, d \tau \, ,
	\ee
	where $f \in L^1 \left(a, \, b \right)$, $b>a$, $f'(t)$ represents the first derivative of $f(t)$, $M(\alpha)$ is a normalization constant such that $M(0) = M(1) = 1$ and $0 < \alpha < 1$.

	A simple and natural generalization of the CF operator can be obtained by replacing the exponential kernel in \eqref{CF-derivative} with a function involving the Mittag-Leffler function \cite{FM-ML} (with one parameter), \ie
	\be \label{ML-1}
	E_\alpha (z) = \sum _{k=0} ^\infty \frac{z^k}{\Gamma (\alpha \, k + 1)} \, , 
	\qquad z, \alpha \in \C, \, \texttt{Re}(\alpha) >0 \, .
	\ee
	
	Indeed, if we consider the kernel
	$$ ^{ABC} \mathcal{K} _\alpha (t) := E_\alpha \left( - \frac{\alpha}{1 - \alpha} \, t^\alpha \right) \, , $$
	this leads to the so called Atangana-Baleanu operator in the Caputo sense (ABC derivative), that reads \cite{ABC}
	\be \label{ABC-derivative} 
	^{ABC} \textbf{D} _{a+} ^\alpha f (t) = 
	\frac{B (\alpha)}{1 - \alpha} \int _a ^t  E_\alpha \left[ - \frac{\alpha}{1 - \alpha} \, (t - \tau) ^\alpha \right] \, f' (\tau) \, d \tau \, ,
	\ee
	where $f \in H^1 \left(a, \, b \right)$, $b>a$, $B(\alpha)$ is a normalization constant such that $B(0) = B(1) = 1$ and $0 < \alpha < 1$, in analogy with the CF case.
	
	It is now important to discuss some aspects concerning the notion of fractional derivative from an historical and terminological perspective, as these observations will turn out to be useful for the concluding remarks. 
	
	The ``old fashioned'' formulation of fractional calculus \cite{Mainardi-Gorenflo-1997, Mainardi-1997} presents the notion of fractional derivative as a concept strongly connected with the Riemann-Liouville fractional integral. The latter, in fact, represent a natural extension (just an analytic continuation) of the well known Cauchy formula for repeated integration, justifying the attribute ``fractional'' appearing in its name.  
	
	According to this perspective, one can proceed to define some fractional derivatives as the left-inverse of the Riemann-Liouville integral. This is the case of the Riemann-Liouville, Caputo, Gr{\"u}nwald-Letnikov and several other derivatives (see \eg \cite{Mainardi_BOOK10, Mainardi-Gorenflo-1997, Mainardi-1997, SKM}). 
	
	On the same line of thought, the attribute ``fractional'' would still fit, in the old fashioned sense, in connection with differential operators representing the left-inverse of a generalization of the Riemann-Liouville integral. This is indeed the case of the Prabhakar derivative \cite{GGPT, GarraGarrappa, Garrappa, RG-FM-GM, PT, ST, GC, HMS, KSS, MG}.

	Nonetheless, it is worth remarking that some more modern criteria for the definition of fractional derivative have been proposed by Ortigueira and Tenreiro Machado in \cite{fractional}, even though the ``old fashioned'' view is already more then enough for supporting the arguments presented in this paper.

\section{Prabhakar function and integral kernel}
	One of the most important generalizations of the Mittag-Leffler function \eqref{ML-1} is the so called Prabhakar function (see \eg \cite{PBK, CGV, GGPT, GarraGarrappa, Garrappa, RG-FM-GM, HMS, KSS, MG, PT, ST}), also known as the three parameters Mittag-Leffler function, which is defined by its series representation as
	\be \label{PBK} 
	E ^\gamma _{\alpha , \beta} (z) = 
	\sum _{k=0} ^\infty \frac{(\gamma) _k}{\Gamma (\alpha \, k + \beta)} \frac{z^k}{k!} \, ,
	\ee 
	where $z \in \C$, $\alpha, \beta, \gamma \in \C$, $\texttt{Re}(\alpha) > 0$, and where $(\gamma) _k$ is the Pochhammer (rising factorial) symbol, that can also be rewritten as $ (\gamma) _k = \Gamma (\gamma + k) / \Gamma (\gamma)$.
	
	Now, starting from \eqref{PBK} it is possible to define a function, known as the Prabhakar kernel and given by
	\be \label{P-kernel} 
	e ^\gamma _{\alpha , \beta} (\omega ; \, t) := t^{\beta - 1} \, E ^\gamma _{\alpha , \beta} (\omega t^\alpha) \, ,
	\ee
	where $t \in \R$, $\alpha, \beta, \gamma, \omega \in \C$ and $\texttt{Re}(\alpha) > 0$.
	
	This allows us to introduce the so called Prabhakar fractional integral \cite{KSS}, \ie
	\be \label{P-integral} 
	 \textbf{E} ^\gamma _{\alpha , \beta , \omega, a+} \, f (t) 
	= \int _a ^t e ^\gamma _{\alpha , \beta} (\omega ; \, t - \tau)  \, 
	f(\tau) \, d \tau
	\, ,
	\ee
	where $0 \leq a < t < b \leq + \infty$ and $f \in L^1 \left(a, \, b\right)$.
	
	The Prabhakar function defined in Eq.~\eqref{PBK} is an entire function of order $\texttt{Re} (\alpha) ^{-1}$ and type one provided that $\alpha, \beta, \gamma \in \C$, $\texttt{Re}(\alpha) > 0$. We can then use this property to prove the following non-trivial result
	\begin{thm} \label{thm-1}
	Let  $f \in L^1 \left(a, \, b\right)$, with $b>a$ and $a, \, b \in \R$. Then,
	\be 
	\textbf{E} ^\gamma _{\alpha , \beta , \omega, a+} \, f (t) = \sum _{k=0} ^\infty \frac{(\gamma) _k \, \omega ^k}{k!} \,  J _{a+} ^{\alpha k + \beta} f(t).
	\ee
	provided that $\alpha, \beta \in \R^+$, $\gamma, \omega \in \C$ and where 
	$$ J_{a+} ^\sigma f(t) := \frac{1}{\Gamma (\sigma)} \int _a ^t f(\tau) \, (t - \tau) ^{\sigma-1} \, d \tau \, , $$
	is the Riemann-Liouville fractional integral, with $\sigma \in \R^+$.
	\end{thm}	
	
	\begin{proof}
	From Eq.~\eqref{PBK} and Eq.~\eqref{P-integral} we have
	\begin{equation*}
	\begin{split}
	\textbf{E} ^\gamma _{\alpha , \beta , \omega, a+} \, f (t) &= 
	\int _a ^t e ^\gamma _{\alpha , \beta} (\omega ; \, t - \tau) \, f(\tau) \, d \tau =\\
	&= \sum _{k=0} ^\infty \frac{(\gamma) _k \, \omega ^k}{\Gamma (\alpha k + \beta) \, k!} 
	\int _a ^t(t - \tau) ^{\alpha k + \beta - 1} \, f(\tau) \, d\tau =\\
	&= \sum _{k=0} ^\infty \frac{(\gamma) _k \, \omega ^k}{k!} \, J_{a+} ^{\alpha k + \beta} f(t)
	\end{split}
	\end{equation*}
	\end{proof}
	
	\section{CF and ABC operators $\&$ Prabhakar fractional integrals}
	In \cite{GC} it was argued that the Caputo-Fabrizio operator can actually be reinterpreted as a simple realization of a Prabhakar fractional integral. Indeed, if we recall that
	$$ E ^1 _{1 , 1} (\omega t) = \exp (\omega \, t) \, ,$$
	therefore
	$$  e ^1 _{1 , 1} (\omega ; \, t) = \exp (\omega \, t) \, , $$
	then one can easily infer that
	\begin{thm} \label{thm-2}
	Let $f \in AC (a, b)$, with $b>a , \, a, b \in \R$ and let $0 < \alpha < 1$, then
	\be 
	^{CF} \textbf{D} _{a+} ^\alpha f (t) = \frac{M(\alpha)}{1 - \alpha} \, \textbf{E} ^1 _{1 , 1 , \omega (\alpha), a+} \,  f' (t) \, .
	\ee
	where $\omega (\alpha) = - \alpha / (1 - \alpha)$.
	\end{thm}	
	
	\begin{proof}
	Indeed, if we consider the definition in Eq.~\eqref{CF-derivative} and the fact that $e ^1 _{1 , 1} (\omega ; \, t) = E ^1 _{1 , 1} (\omega t) = \exp (\omega \, t)$, we have
	\be 
	\notag
	^{CF} \textbf{D}_{a+} ^\alpha f (t) &\!\!=\!\!& 
	\frac{M(\alpha)}{1 - \alpha} \int _a ^t  \exp \left[ - \frac{\alpha}{1 - \alpha} \, (t - \tau) \right] \, f' (\tau) \, d \tau =\\ \notag
	&\!\!=\!\!& 
	\frac{M(\alpha)}{1 - \alpha} \int _a ^t  E ^1 _{1 , 1} \left[ - \frac{\alpha}{1 - \alpha} \, (t - \tau) \right] \, f' (\tau) \, d \tau =\\ \notag
	&\!\!=\!\!&
	\frac{M(\alpha)}{1 - \alpha} \int _a ^t  e ^1 _{1 , 1} \left(- \frac{\alpha}{1 - \alpha} ; \, t - \tau\right) \, f' (\tau) \, d \tau =\\ \notag
	&\!\!=\!\!& \frac{M(\alpha)}{1 - \alpha} \, \textbf{E} ^1 _{1 , 1 , \omega (\alpha), a+} \,  f' (t) \, .
	\ee  
	\end{proof}
	
	\begin{rmk}
	Here $AC (a, b)$ represents the class of absolutely continuous functions. For these kind of functions it is important to recall that if $f \in AC (a, b)$ then $f$ is differentiable on $(a, b)$ and $f' \in L^{1} (a, b)$.
	\end{rmk}
	
	A similar argument can be presented also for the Atangana-Baleanu operator (in the Caputo sense).
	
	\begin{thm} \label{thm-3}
	Let $f \in AC (a, b)$, with $b>a , \, a, b \in \R$ and let $0 < \alpha < 1$, then
	\be
	^{ABC} \textbf{D} _{a+} ^\alpha f (t) = \frac{B(\alpha)}{1 - \alpha} \, \textbf{E} ^1 _{\alpha , 1 , \omega (\alpha), a+} \,  f' (t) \, .
	\ee
	where, again, $\omega (\alpha) = - \alpha / (1 - \alpha)$.
	\end{thm}
	
	\begin{proof}
	Indeed, given that
	$$ e ^1 _{\alpha , 1} (\omega ; \, t) =  E_\alpha (\omega \, t ^\alpha) \, ,  $$
	it is not difficult to see that
	\be 
	\notag
	^{ABC} \textbf{D} _{a+} ^\alpha f (t) &\!\!=\!\!& 
	\frac{B (\alpha)}{1 - \alpha} \int _a ^t  E_\alpha \left[ - \frac{\alpha}{1 - \alpha} \, (t - \tau) ^\alpha \right] \, f' (\tau) \, d \tau =\\ \notag
	&\!\!=\!\!&
	\frac{B(\alpha)}{1 - \alpha} \int _a ^t  e ^1 _{\alpha , 1} \left(- \frac{\alpha}{1 - \alpha} ; \, t - \tau\right) \, f' (\tau) \, d \tau =\\ \notag
	&\!\!=\!\!&
	\frac{B(\alpha)}{1 - \alpha} \, \textbf{E} ^1 _{\alpha , 1 , \omega (\alpha), a+} \,  f' (t) \, .
	\ee 
	\end{proof}

	Using these results, it is now interesting to see if one can infer something about the fractional nature of the CF and ABC operators and whether they are really needed on their own.
	
	Starting with the Caputo-Fabrizio operator it is easy to see that
	
	\begin{thm} \label{thm-4}
	Let $f \in AC (a, b)$, with $b>a , \, a, b \in \R$ and let $0 < \alpha < 1$, then
	\begin{equation} \label{eq-thm-4}
	\begin{split}
	^{CF} \textbf{D} _{a+} ^\alpha f (t) = &- \frac{M(\alpha) \, f(a^+)}{1 - \alpha} \, \exp \left[ - \frac{\alpha \, (t - a)}{1 - \alpha} \right] \\
	 &+ \frac{M(\alpha)}{1 - \alpha} \, \sum _{k = 0} ^\infty \left(- \frac{\alpha}{1 - \alpha}\right)^k \, J _{a+} ^k f(t) \, .
	\end{split}
	\end{equation}
	\end{thm}
	
	\begin{proof}
	Due to Theorem~\ref{thm-2}, one just need to observe that
	\be 
	^{CF} \textbf{D}_{a+} ^\alpha f (t) = \frac{M(\alpha)}{1 - \alpha} \, \textbf{E} ^1 _{1 , 1 , \omega (\alpha), a+} \,  f' (t) \, ,
	\ee
	with $\omega (\alpha) = - \alpha / (1 - \alpha)$.
	
	Therefore, because of Theorem~\ref{thm-1}
	\be \label{aux}
	^{CF} \textbf{D} _{a+} ^\alpha f (t) = \frac{M(\alpha)}{1 - \alpha} \, 
	\sum _{k=0} ^\infty \omega (\alpha) ^k \, J _{a+} ^{k + 1} f' (t) \, ,
	\ee
	where we have used the fact that $(1)_k = k!$.
	
	Now, recalling that 
	$$ J _{a+} ^{k + 1} f' (t) = J _{a+} ^{k} \left[ f (t) - f(a) \right] = J _{a+} ^{k} f (t) - f(a) \, \frac{(t - a)^k}{k!} \, , $$
	from which one can easily connect Eq.~\eqref{aux} with Eq.~\eqref{eq-thm-4}.
	\end{proof}
	
	This result tells us that the CF operator is nothing but an infinite linear combination of ordinary (integer power) repeated integrals of the function $f(t)$ and, therefore, it does not lead to a fractional behaviour in linear systems whose dynamics is governed by such an operator.
	
	A similar result can also be obtained for the ABC operator, indeed one finds that
	\begin{thm} \label{thm-5}
	Let $f \in AC (a, b)$, with $b>a , \, a, b \in \R$ and let $0 < \alpha < 1$, then
	\begin{equation} \label{eq-thm-5}
	\begin{split}
	^{ABC} \textbf{D}_{a+} ^\alpha f (t) = &- \frac{B(\alpha) \, f(a^+)}{1 - \alpha} \, 
	E_\alpha \left[ - \frac{\alpha \, (t - a) ^\alpha}{1 - \alpha} \right] \\
	 &+ \frac{B(\alpha)}{1 - \alpha} \, \sum _{k = 0} ^\infty \left(- \frac{\alpha}{1 - \alpha}\right)^k \, J_{a+} ^{\alpha \, k} f(t) \, .
	\end{split}
	\end{equation}
	\end{thm}
	
	\begin{proof}
	As above, from Theorem~\ref{thm-3} one has that
	\be 
	^{ABC} \textbf{D}_{a+} ^\alpha f (t) = \frac{B(\alpha)}{1 - \alpha} \, \textbf{E} ^1 _{\alpha , 1 , \omega (\alpha), a+} \,  f' (t) \, ,
	\ee
	with $\omega (\alpha) = - \alpha / (1 - \alpha)$.
	
	Thus,
	\begin{equation}
	\begin{split}
	^{ABC} \textbf{D} _{a+} ^\alpha f (t) &= \frac{B(\alpha)}{1 - \alpha} \, \sum _{k=0} ^\infty \omega (\alpha) ^k \, J _{a+} ^{\alpha k + 1} f' (t) =\\
	&= \frac{B(\alpha)}{1 - \alpha} \, \sum _{k=0} ^\infty \omega (\alpha) ^k \, J _{a+} ^{\alpha k} \left[ f (t) - f(a) \right] =\\
	&= \frac{B(\alpha)}{1 - \alpha} \, \sum _{k=0} ^\infty \omega (\alpha) ^k \left[ J _{a+} ^{\alpha k} f (t) - f(a) \, 
	\frac{(t - a) ^{\alpha k}}{\Gamma (\alpha k + 1)} \right] \, ,
	\end{split}
	\end{equation}
	from which we can recover Eq.~\eqref{eq-thm-5} just by recalling the series expansion of the one-parameter Mittag-Leffler function. 	
	\end{proof}
	
	This last theorem tells us that, after all, the ABC operator is ultimately a fractional one. However, it is nothing but an infinite series involving Riemann-Liouville integrals of order $\alpha \, k$, with $k \in \N$. This suggests that a linear system expressed in terms of such an operator leads to a fractional dynamics formally equivalent to a model involving a sort of combination of Caputo fractional derivatives.
	
	\section{On Differential Equations involving the CF and ABC operators}
	In the previous section we used an interesting property of Prabhakar integrals to hint at the fact that the CF and ABC operators might not represent a significant improvement for the literature on the theory of fractional calculus. In this section we further support this statement by analysing the structure of a general differential equation of the form
	\be \label{eq-fde}
	\mathcal{D} ^\alpha y (t) = F [t, \, y (t)] \, ,
	\ee
	with $F [t, \, y (t)]$ being a sufficiently regular functional of $t>0$ and $y(t)$, and $\mathcal{D} ^\alpha$ representing either the CF or the ABC operator for $\alpha \in (0, 1)$ and $a = 0$.
	
	By means of the Laplace transform method one can easily prove the following results:
	
	\begin{thm} \label{Thm-fde-1}
	Let us consider the differential equation
	\be \label{eq-fde1}
	^{CF} \textbf{D} _{0+} ^\alpha y (t) = F [t, \, y (t)] \, ,
	\ee
	with $\alpha \in (0, 1)$.
	
	Then, we can recast the latter in the following form,
	\be \label{eq-fde1-1}
	\frac{dy (t)}{dt} = \frac{1-\alpha}{M (\alpha)} \, \frac{d}{dt} F [t, \, y (t)] + \frac{\alpha}{M (\alpha)} \, F [t, \, y (t)] + 
	\frac{1-\alpha}{M (\alpha)} \,F [0^+, \, y (0^+)] \, \delta (t) \, ,
	\ee
	which is an ordinary differential equation of the first order on the space of tempered distribution over $\R^+$.
	\end{thm}
	
	\begin{proof}
	First, let us recall the Laplace transform for the CF operator,
	$$
	\mathcal{L} \left\{ ^{CF} \textbf{D} _{0+} ^\alpha f (t)  \, ; \, s \right\} = 
	\frac{M (\alpha)}{1 - \alpha} \frac{s \, \wt{f} (s) - f(0^+)}{s + \frac{\alpha}{1-\alpha}} \, ,
	$$
	with $\wt{f} (s) \equiv \mathcal{L} [f(t) \, ; \, s]$ representing the Laplace transform of the function $f(t)$.
	
	If the Laplace transform the applied to both sides of Eq.~\eqref{eq-fde1}, one gets
	\be \label{eq-fde1-2}
	\frac{M (\alpha)}{1 - \alpha} \frac{s \, \wt{f} (s) - f(0^+)}{s + \frac{\alpha}{1-\alpha}} = \wt{\Phi} (s) \, ,
	\ee
	where we denoted $\Phi (t) \equiv F [t, \, y (t)]$.
	
	Now, one can easily manipulate algebraically Eq.~\eqref{eq-fde1-2} and recast it as follows
	\be \label{eq-fde1-3} \notag
	M (\alpha) \, [s \, \wt{y} (s) - y(0^+)] &\!\!=\!\!& (1 - \alpha) \, \left( s + \frac{\alpha}{1-\alpha} \right) \, \wt{\Phi} (s) \\
	&\!\!=\!\!& (1 - \alpha) \, s \, \wt{\Phi} (s) + \alpha \, \wt{\Phi} (s) \\ \notag
	&\!\!=\!\!& (1 - \alpha) \, [s \, \wt{\Phi} (s) - \Phi (0^+)] + (1 - \alpha) \, \Phi (0^+) + \alpha \, \wt{\Phi} (s) \, .
	\ee
	
	Hence, inverting back to the time domain and recalling that
	$$ \mathcal{L} [f' (t) \, ; \, s] = s \, \wt{f} (s) - f(0^+) \, \, , \qquad \mathcal{L} [\delta (t) \, ; \, s] = 1 \,\, , $$
	then we get exactly Eq.~\eqref{eq-fde1-1}.
 	\end{proof}
 	
 	\begin{thm} \label{Thm-fde-2}
	Let us consider the differential equation
	\be \label{eq-fde2}
	^{ABC} \textbf{D} _{0+} ^\alpha y (t) = F [t, \, y (t)] \, ,
	\ee
	with $\alpha \in (0, 1)$.
	
	Then, we can recast the latter in the following form,
	\be \label{eq-fde2-1}
	^{C} \textbf{D}_{0+} ^\alpha y (t) &\!\!=\!\!& \frac{1-\alpha}{B (\alpha)} \, {}^{C} \textbf{D} _{0+} ^\alpha F [t, \, y (t)] + 
	\frac{\alpha}{B (\alpha)} \, F [t, \, y (t)] + \\ \notag
	&\!\!+\!\!&
	\frac{(1-\alpha) \, t^{-\alpha}}{B (\alpha) \, \Gamma (1-\alpha)} \,F [0^+, \, y (0^+)] \, ,
	\ee
	which is a fractional differential equation of order on $\alpha \in (0, 1)$ over $\R^+$ that involves only Caputo factional derivatives (\ie $^{C} \textbf{D} _{0+} ^\alpha f (t) := J _{0+} ^{1 - \alpha} f' (t)$ for $\alpha \in (0, 1)$).
	\end{thm}
		
	\begin{proof}
	One shall proceed in analogy with the proof of Theorem~\ref{Thm-fde-1}. Specifically, one should recall that
	$$ \mathcal{L} [^{C} \textbf{D} _{0+} ^\alpha f (t) \, ; \, s] = s^\alpha \, \wt{f} (s) - s^{\alpha-1} \, f(0^+) \, , $$
	for $\alpha \in (0,1)$, and that
	$$\mathcal{L} [ t^p \, ; \, s] = \frac{\Gamma (p + 1)}{s ^{p+1}} \, , \quad p>-1 \, . $$
 	\end{proof}	
 	
 	Equivalently, by recasting the expressions in Eq.~\eqref{eq-fde1-3} (or just by integrating both sides of Eq.~\eqref{eq-fde1-2}), one can also show that Eq.~\eqref{eq-fde1} is formally equivalent to
 	\be 
 	y(t) = y(0^+) + \frac{1-\alpha}{M(\alpha)} \, F [t, \, y (t)] + \frac{\alpha}{M(\alpha)} \int _0 ^t F [\tau, \, y (\tau)] \, d \tau \, ,
 	\ee
 	which is nothing but an (ordinary) integral equation for the function $y(t)$.
 	
 	By following a similar procedure for the ABC case, one can easily find that Eq.~\eqref{eq-fde2} is formally equivalent to
 	\be 
	y (t) = y(0^+) + \frac{1-\alpha}{B (\alpha)} \, F [t, \, y (t)] + 
	\frac{\alpha}{B (\alpha)} \, J _{0+} ^{\alpha} F [t, \, y (t)] 
	\, ,
 	\ee
that, again, is nothing but a fractional integral equation for the function $y(t)$ involving only the Riemann-Liouville fractional integral.

	\section{Physical remarks: the Scott-Blair model} \label{sec:sb}
	Let us now discuss some physical implications of the application of these alleged new fractional operators in linear viscoelasticity. In particular, let us consider as a working example the so called Scott-Blair model \cite{Mainardi_BOOK10, MS}, \ie
	\be \label{SB-eq}
	\sigma (t) = \eta \, ^{C} \textbf{D} _{0+} ^\alpha \varepsilon (t) \, , \qquad \alpha \in (0, 1) \, ,
	\ee
	with $\sigma$ and $\varepsilon$ denoting respectively the stress and strain function for the considered material, $\eta > 0$ being the viscosity term and $t>0$ representing the non-dimensional time variable (\ie the time rescaled by the characteristic time-scale of the system).
	
	Recalling that, in linear viscoelasticity, the relaxation modulus $G(t)$ has the following property \cite{Mainardi_BOOK10, MS}, 
	$$ \wt{\sigma} (s) = s \, \wt{G} (s) \, \wt{\epsilon} (s) \, , $$
	in the Laplace domain, then it is easy to see that
	\be \label{SB-G}
	G_{\tiny \mbox{Scott-Blair}} (t) = \frac{\eta}{\Gamma (1 - \alpha)} \, t^{- \alpha} \, ,
	\ee
	assuming $\varepsilon (0^+) = 0$ (see \eg \cite{AG-FCAA, Mainardi_BOOK10, MS}).
	
	It is now interesting to see what happens if we replace the Caputo fractional derivative with the CF or the ABC operator.
	
	Employing the CF operator in Eq.~\eqref{SB-eq}, again assuming $\varepsilon (0^+) = 0$, one can easily compute the relaxation modulus of the resulting model in the Laplace domain, \ie
	\be 
	\wt{G} _{CF} (s) = \frac{\eta \, M(\alpha)}{1 - \alpha} \, \left( s + \frac{\alpha}{1 - \alpha} \right)^{-1} \, ,
	\ee
	that, once inverted back to the time domain gives
	\be \label{CF-G}
	G_{CF} (t) = \frac{\eta \, M(\alpha)}{1 - \alpha} \, \exp \left( - \frac{\alpha}{1 - \alpha} \, t \right) \, ,
	\ee
	from which we can infer, by comparing this relaxation modulus with the well established results discussed in \cite{Mainardi_BOOK10, MS}, that a Scott-Blair model involving the CF operator is nothing but an ordinary Maxwell model with a rescaled glass modulus (\ie $G(0^+)$) and a relaxation time rescaled by a factor $(1 - \alpha)/\alpha$.
	
	Analogously, if we perform the same analysis for a Scott-Blair model involving the ABC operator, one can easily conclude that
	\be \label{AB-G}
	G_{ABC} (t) = \frac{\eta \, B(\alpha)}{1 - \alpha} \, E_\alpha \left( - \frac{\alpha}{1 - \alpha} \, t ^\alpha \right) \, ,
	\ee
	that, again, is nothing but a fractional Maxwell model of order $\alpha$ with rescaled glass modulus and relaxation time.

	A pictorial comparison between the relaxation moduli obtained from the traditional Scott-Blair model and the results displayed in Eq.~\eqref{CF-G} and Eq.~\eqref{AB-G} is provided in \figurename~\ref{Fig-1}. Specifically, from this plot we can immediately infer that $G_{ABC} (t)$ clearly shows a short-time stretched exponential decay (with an infinite negative derivative) that transitions continuously to a power low decay for long times, which is the main feature of the fractional Maxwell model. Whereas, if we consider the behaviour of $G_{CF} (t)$ one can clearly see that, while the Scott-Blair model involving the ABC operator can somehow resemble the typical Scott-Blair behaviour for long times ($\sim t^{-\alpha}$), the model with the CF operator features a purely exponential fall-off, which is what one would expect from a purely ordinary theory. Besides, as stressed above, the relaxation modulus $G_{CF} (t)$ corresponds to the one of the ordinary Maxwell model with rescaled glass modulus and relaxation time.
		
	\begin{figure}[h!]
	\centering
	\includegraphics[scale=0.5]{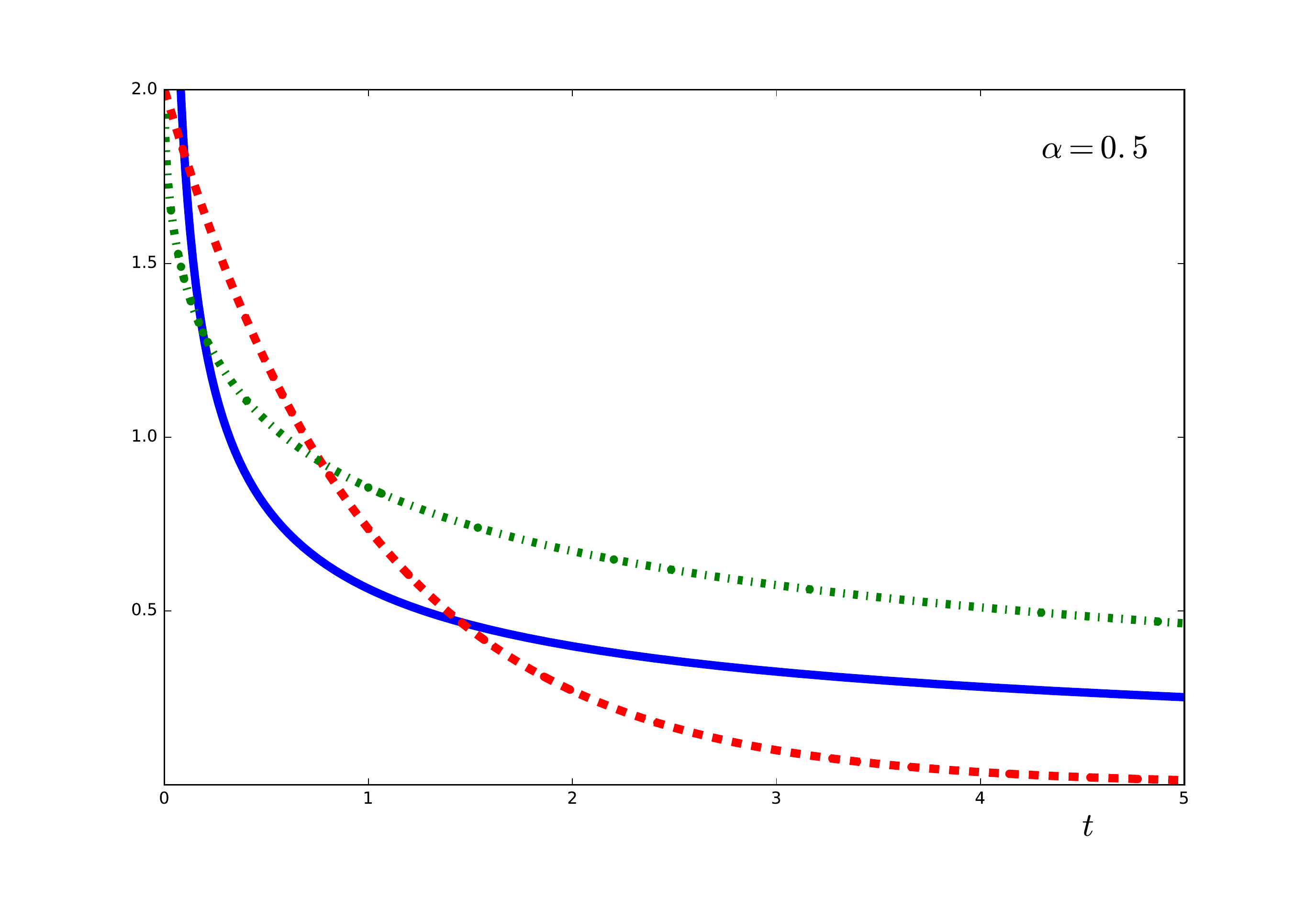}
	\caption{Comparison between the Scott-Blair relaxation modulus (solid line), $G_{CF} (t) / M(\alpha)$ (dashed line) and $G_{ABC} (t) / B(\alpha)$ (dotted line), for $\eta = 1$ and $\alpha = 0.5$. } \label{Fig-1}
	\end{figure}		
	
	Consequently, from the results of this section one can fairly conclude that, even from a purely physical perspective, the CF and ABC differential operators do not seem to add anything new to the realm of applications of fractional calculus to physical systems.
		
	\section{Concluding remarks}	
	
	In light of the old feshoned definition of fractional derivative, one can infer that the CF and ABC operators do not meet any of the requirements for them to be referred to as new fractional derivatives. In particular, in view of Theorem~\ref{thm-4} and \ref{Thm-fde-1} one can definitely state that the CF is neither a new definition of fractional derivative nor a fractional operator, as it was also stressed in \cite{OM, Tarasov}. Whereas, from Theorem~\ref{thm-5} and \ref{Thm-fde-2}, one can conclude that the ABC operator has still some resemblance of a fractional nature, even thought it does not provide any real improvement with respect to the theory of fractional calculus based on the Caputo derivative or, more generally, with respect to the Prabhakar calculus.

	Furthermore, in Section~\ref{sec:sb} we have shown that these two operators do not seem to add any new insight to the (already well established) theory of linear viscoelasticity, as they induce a mere reformulation of some well known results when employed in the constitutive equation for a given material.
		
\section*{Acknowledgments}	
	The work of the authors has been carried out in the framework of the activities of the National Group of Mathematical Physics (GNFM, INdAM).	
	
	Moreover, the work of A.G. has been partially supported by \textit{GNFM/INdAM Young Researchers Project} 2017 ``Analysis of Complex Biological Systems''.

\section*{\textbf{Note Added}}
	After the publication this paper, we became aware of the work by Fernandez and Baleanu \cite{FB} in which the authors provide an alternative derivation of Eq.~\eqref{eq-thm-5}. We thank Abdon Atangana for bringing this reference to our attention.
	
	The independent derivation presented here follows directly, as a very particular case, from a general result concerning Prabhakar integrals (\ie Theorem~\ref{thm-1}, introduced and proved in this manuscript).
	

\end{document}